\documentclass[12pt]{article}
\usepackage{latexsym, amsmath, amsfonts, amsthm, amssymb}
\usepackage{times}
\usepackage{a4wide}
\usepackage{graphicx, tocloft}
\usepackage[dvipsnames]{xcolor}

\def \RR {\mathbb R}

\def \EE {\mathbb E}

\def \PP {\mathbb P}

\def \cT {\mathcal T}

\def \cS {\mathcal S}

\def \cK {\mathcal K}

\newtheorem{theorem}{Theorem}[section]

\newtheorem{lemma}[theorem]{Lemma}

\newtheorem{proposition}[theorem]{Proposition}
\newtheorem{corollary}[theorem]{Corollary}

\theoremstyle{definition}

\def\qed{\hfill $\vcenter{\hrule height .3mm
		\hbox {\vrule width .3mm height 2.1mm \kern 2mm \vrule width .3mm
			height 2.1mm} \hrule height .3mm}$ \bigskip}

\begin{document}

\title{Linear extension operators for Sobolev spaces on radially-symmetric binary trees}

\author{}
\date{}
\maketitle
\vspace{-50pt}

\begin{center}
{\large Charles Fefferman\textsuperscript{1} and Bo'az  Klartag\textsuperscript{2}}

\bigskip\bigskip
{\it Dedicated in friendship to David Jerison}
\end{center}

\footnotetext[1]{Department of Mathematics, Princeton University, Fine Hall, Washington Road, Princeton, New Jersey 08544, USA.
Email: cf@math.princeton.edu. Supported
by the Air Force Office of Scientific Research, grant number FA9950-18-1-0069
and the National Science Foundation (NSF), grant number DMS-1700180.}
\footnotetext[2]{Department of Mathematics, Weizmann Institute of Science, Rehovot 7610001, Israel.
Email: boaz.klartag@weizmann.ac.il.  Supported by a grant from the Israel Science Foundation (ISF). \\ $ $ 
AMS classification (MSC 2020): 46E36
}

\begin{abstract}
{
Let $1 < p < \infty$ and suppose that we are given a function $f$ defined on the leaves of a weighted tree.  We would like to extend
$f$ to a function $F$ defined on the entire tree, so as to minimize the weighted $W^{1,p}$-Sobolev norm of the extension.
An easy situation is when $p = 2$, where the harmonic extension operator provides such a function $F$.
In this note we record our analysis of the particular case of a radially-symmetric binary tree, which is a complete, finite, binary tree with weights that depend only on the distance from the root.
Neither the averaging operator nor the harmonic extension operator work here in general. Nevertheless, we prove
the existence of a linear extension operator whose norm is bounded by a constant depending solely on $p$. This operator is a variant of the standard harmonic extension operator,
and in fact it is harmonic extension with respect to a certain Markov kernel determined by $p$ and by the weights.}
\end{abstract}

\section{Introduction}
\label{sec1}

Consider a full binary tree of height $N$, whose set of vertices is denoted by
$$ V = \bigcup_{k=0}^N \{0,1 \}^k, $$
i.e., the vertices are strings of zeroes and ones of length at most $N$.
For $x \in \{0, 1 \}^k$ and $\ell \leq k$ we write $\pi_\ell(x) \in \{0,1 \}^{\ell}$
for the prefix of $x$ of length $\ell$.
Thus for $k \geq 1$, the parent of a vertex $x \in \{0,1 \}^k \subseteq V$ is the vertex
$\pi_{k-1}(x)$. The set $\{ 0, 1 \}^0$ is a singleton whose unique element is denoted by $\emptyset$, the empty string, which is the root of the tree. The set of leaves of the tree is $\{0,1 \}^N$, and all other vertices in $V$ are internal vertices. A vertex $x \in \{ 0, 1 \}^k$ is said to have depth
$$ d(x) = k, $$
thus leaves have depth $N$ and the root has depth $0$. The collections of bijections from $V$ to $V$ that preserve depth and parenthood relations  form a group. This group is referred to as the symmetry group of the tree. It has $2^{2^N-1}$ elements, and it is a 2-Sylow subgroup of the group of all permutations of the leaves.
The set  $$ E = \bigcup_{k=1}^N \{0,1 \}^k = V \setminus \{ \emptyset \} $$
is referred to as the set of {\it edges} of the tree, and the depth of an edge $e \in \{0, 1 \}^k$ is $d(e) = k$.
That is, we think of $e \in E$ as an undirected edge connecting the vertex whose string corresponds to $e$ to its unique parent.
Each internal vertex other than the root is connected to three vertices, which are its parent and its two children.
Assume that we are given edge weights
$$ W_1,W_1,\ldots,W_{N} > 0, $$
where we view $W_k$ as the weight of all edges of depth $k$.
For $1 < p < \infty$, the associated $\dot{W}^{1,p}$-seminorm is defined, for $F: V \rightarrow \RR$, via
\begin{equation} \| F \|_{\dot{W}^{1,p}(V)} =  \left( \sum_{k=1}^{N} W_k \sum_{x \in \{0,1 \}^{k}} |F(x) - F(\pi_{k-1}(x))|^p \right)^{1/p}. \label{eq_1802} \end{equation}
We write $\partial V = \{ 0, 1 \}^N \subseteq V$ for the set of leaves of the tree. The trace of the $
\| \cdot \|_{\dot{W}^{1,p}} $-seminorm is defined, for $f: \partial V \rightarrow \RR$, via
\begin{equation}  \| f \|_{\dot{W}^{1,p}(\partial V)} = \inf \left \{ \| F \|_{\dot{W}^{1,p}(V)} \, ; \, F|_{\partial V} = f \right \},
\label{eq_1210} \end{equation}
i.e., the infimum of the $\dot{W}^{1,p}$-seminorm over all extensions of $f$ from the leaves to the entire tree.
We write $\RR^V$ for the collection of all functions $f: V \rightarrow \RR$, and similarly $\RR^{\partial V}$ is
the collection of all functions $f: \partial V \rightarrow \RR$. Our main result is the following:

\begin{theorem} Let $1 < p < \infty$ and let $W_1,\ldots,W_N > 0$. Then there exists a linear operator $H: \RR^{\partial V} \rightarrow \RR^V$ with the following properties:
\begin{enumerate}
  \item It is a linear extension operator, i.e., $(H f)(x) = f(x)$ for any $x \in \partial V$ and any function $f : \partial V \rightarrow \RR$.
  \item Its norm is bounded by a constant $\bar{C}_p$ depending only on $p$, i.e., for any $f: \partial V \rightarrow \RR$,
  $$ \| H f \|_{\dot{W}^{1,p}(V)} \leq \bar{C}_p
  \| f \|_{\dot{W}^{1,p}(\partial V)}. $$
\end{enumerate}
In fact, we have the bound
\begin{equation}  \bar{C}_p \leq
	4 p^{1/p} q^{1/q} \cdot \left( 1 + \max \{ (p-1)^{-1/p}, (q-1)^{-1/q}  \} \right) \leq C \cdot \max \left \{ \frac{1}{p-1}, \frac{1}{q-1} \right \}, \label{eq_1003} \end{equation}
where $q = p / (p-1)$ and where $C > 0$ is a universal constant.
\label{thm_1226_}
\end{theorem}

The proof of Theorem \ref{thm_1226} is constructive, and the extension operator $H$ that we construct is in fact a {\it harmonic extension operator} with respect to a certain random walk defined on the tree. At each step the random walk jumps from a vertex to one of its neighbors, where of course the neighbors
of a vertex are its parent and its children.
The Markov kernel corresponding to the random walk is invariant under the symmetries of the tree, thus the probability to move from a vertex to its neighbor depends only on the weights of the vertex and of its neighbor.
The Markov kernel of our random walk is determined by the following requirement: For any $s=1,\ldots,N$,
the probability that a random walk starting at some vertex of depth $s$ will reach a leaf before reaching a vertex of depth $s-1$ equals
\begin{equation}  q_{s} = \frac{(2^s W_{s})^{-1/(p-1)}}{\sum_{k=s}^{N} (2^{k} W_{k})^{-1/(p-1)}}.
\label{eq_338} \end{equation}
Thus the weights of our random walk typically depend on $p \in (1, \infty)$, except for the case where $W_s$ is proportional to $2^{-s}$. This seems inevitable. Indeed, in some examples such as the $3^{rd}$ example in Section 2 below, the linear
extension operator $H: \RR^{\partial V} \rightarrow \RR^V$ that corresponds to the parameter
value $p = p_0$, is not uniformly bounded for any  $p \in (1, \infty) \setminus \{ p_0 \}$. When $p = 2$,
our random walk coincides with the usual random walk corresponding to the given weights on the edges of the binary tree, and thus in this case
 $H$ is the standard harmonic extension operator
(hence $\bar{C}_p = 1$ for $p = 2$).

\medskip There is a certain range of weights where the {\it averaging operator} yields a uniformly
bounded linear extension operator, as proven by Bj\"orn $\times 2$, Gill and Shanmugalingam \cite{BBGS}. The averaging operator
is the extension operator that assigns to each internal vertex $v$ the average of the function
values on the leaves of the subtree whose root is $v$. This averaging operator seems
natural also from the point of view of Whitney's extension theory,
see the work by Shvartsman \cite{Shvartsman} on Sobolev extension in $W^{1,p}(\RR^n)$.
However, there are examples of radially-symmetric  binary trees where the averaging operator does not provide
a uniformly bounded operator, such as the case where $W_k = 2^{-k}$ for all $k$.

\medskip What about trees with weights that are not radially-symmetric? Suppose that the edge weights are arbitrary positive
numbers $(W_e)_{e \in E}$ that are not necessarily determined by the depth of the edge. When $p=2$,
there is still a harmonic extension operator of norm one from $\dot{W}^{1,p}(\partial V)$ to $\dot{W}^{1,p}(V)$.
However, for $p \neq 2$ the situation seems  subtle. We conjecture that
in the general case of non-radially-symmetric  tree weights, there is no linear extension operator whose norm is bounded by a function of $p$ alone.
This conjecture is closely related to questions about well-complemented subspaces of $\ell_p^n$ that are beyond the scope of
this note.

\medskip In order to prove Theorem \ref{thm_1226_} we reformulate
the problem in a way that brings us closer to analysis in $\ell_p$-spaces.
For $1 < p < \infty$, the associated $L^p(E)$-norm is
defined, for $f: E \rightarrow \RR$, via
\begin{equation}  \| f \|_p = \| f \|_{L^p(E)} = \left( \sum_{k=1}^{N} W_k \sum_{x \in \{0,1 \}^{k}} |f(x)|^p \right)^{1/p}.
\label{eq_346} \end{equation}
We think of  a function $f: E \rightarrow \RR$
as the gradient of a function $\tilde{f}: V \rightarrow \RR$, uniquely determined up to an additive constant.
Given $f: E \rightarrow \RR$ we thus define a function $\tilde{f}: V \rightarrow \RR$ as follows:
For any $x \in V$ other than the root,
\begin{equation}  \tilde{f}(x) = \sum_{i=1}^{d(x)} f(\pi_i x), \label{eq_1756_}
\end{equation}
while $\tilde{f}(x) = 0$ if $x = \emptyset$ is the root. The only property of $\tilde{f}$ that matters  is that for all $x \in E$,
$$ f(x) = \tilde{f}(x) - \tilde{f}(\pi_{d(x)-1}(x)). $$
We are interested in finding a linear operator $T: L^p(E) \rightarrow L^p(E)$, with a uniform bound on its operator norm, that has the following
properties:
\begin{enumerate}
\item The operator $T$ takes the form
\begin{equation}  T f(x)  = \tilde{T} \tilde{f}(x) - \tilde{T} \tilde{f}(\pi_{d(x)-1} x)
\label{eq_552} \end{equation}
for some linear operator $\tilde{T}: \RR^V \rightarrow \RR^V$. That is, $\tilde{T}$ takes functions on $V$ to functions on $V$,
and $T$ is induced from $\tilde{T}$ via  formula (\ref{eq_552}).
\item The operator $\tilde{T}$ is equivariant
with respect to the tree symmetries and it satisfies $\tilde{T}(1) \equiv 1$, i.e., it maps the constant function $1$ to itself.
\item The function $\tilde{T} g$ coincides with the function $g$ on the leaves of the tree, i.e. $(\tilde{T} g)|_{\partial V} = g|_{\partial V}$ for any $g: V \rightarrow \RR$.
\item The function $\tilde{T} g$ is determined by the values of the function $g$ on the leaves of the tree.
\end{enumerate}

The operator norm of $\tilde{T}$ with respect to the $\dot{W}^{1,p}$-seminorm
equals to the operator norm of $T$ with respect to the $L^p(E)$-norm.
Defining $H (f|_{\partial V})= \tilde{T} f$, Theorem \ref{thm_1226_} may thus be reformulated as follows:

\begin{theorem} Let $1 < p < \infty$ and let $W_1,\ldots,W_N > 0$. Then there exists a linear operator $T: L^p(E) \rightarrow L^p(E)$ with the above properties, whose operator norm is at most
	a certain constant $\bar{C}_p$ depending only on $p$.
In fact, we have the bound (\ref{eq_1003}) for the constant $\bar{C}_p$.
\label{thm_1226}
\end{theorem}

The proof of Theorem \ref{thm_1226} occupies the next three sections. In Section \ref{sec2} we
discuss invariant random walks on a full binary tree and describe the corresponding
harmonic extension operator. In Section \ref{sec3} we deal with the problem of bounding
the norm of this operator, and use the symmetries of the problem in order to reduce it to a one-dimensional question.
This one-dimensional question is then answered in Section \ref{sec4} using the Muckenhoupt criterion \cite{M}.

\medskip When analyzing the binary tree we use the following notation:
We write $dlca(x,y)$ for the length of the maximal prefix shared by the strings $x \in \{0,1\}^k$ and $y \in \{0,1 \}^\ell$,
while $lca(x,y)$ is the maximal prefix itself. Thus for two vertices $x,y \in V$, their
least common ancestor is $lca(x, y) \in V$ and its depth is $dlca(x,y) \in \{ 0,\ldots,N \}$. Note  that for any $x \in E$ and $\omega \in \partial V$,
\begin{equation}
 dlca(\pi_{d(x)-1} x, \omega) = \min \{ d(x) - 1, dlca(x, \omega) \}.
 \label{eq_1719}
\end{equation}

\medskip
\emph{Acknowledgements.} We would like to thank
Jacob Carruth, Arie Israel, Anna Skorobogatova and
Ignacio Uriarte-Tuero for helpful conversations. This research was conducted
while BK was visiting Princeton University's Department of Mathematics; he is grateful for their gracious hospitality.

\section{Invariant random walks}
\label{sec2}

A Markov chain on $V$ is a sequence of random variables $R_1,R_2,\ldots \in V$ such that the distribution of $R_{i+1}$
 conditioned on $R_1,\ldots,R_i$ is the same as its distribution conditioned on $R_{i}$. A Markov chain is time-homogeneous
 if for any $x,y \in V$, the probability that $R_{i+1} = x$ conditioned on the event $R_{i} = y$ does not depend on $i$.
A random walk on $V$ is a time-homogeneous Markov chain $R_1,R_2,\ldots \in V$ such that $R_{i+1}$ is a neighbor of $R_i$ with probability one.

\medskip
We say that the random walk is {\it invariant} if the probability to jump from a vertex $x$ to a vertex $y$ depends only on the depths $d(x)$ and $d(y)$.
Our random walk will be invariant, and it will stop when it reaches a leaf, i.e., we have the stopping time
$$ \tau = \min \{i \geq 1 \, ; \, R_i \ \text{is a leaf}  \}. $$
For $s \geq 1$ we define $q_s$ to be the probability of the following event: Assuming that $R_1$ is a vertex of depth $s$,
the event is that $R_i$ will remain at the subtree whose root is $R_1$ for all $1 \leq i \leq \tau$.
Equivalently,
define $$ X_i = d(R_i) \in \{0,\ldots, N \}. $$
Then $X_1,X_2,\ldots \in \{0, \ldots, N \}$ is a random walk, since $|X_{i+1} - X_i| = 1$ for all $i$. Furthermore,
\begin{equation}  q_s = \PP (\forall 1 \leq i \leq \tau, X_i \geq s \, | \, X_1 = s). \label{eq_X} \end{equation}
Clearly
\begin{equation} q_0 = q_N = 1. \label{eq_322} \end{equation}
For $r,s \in \{0,\ldots,N \}$ with $r \leq s$ we set
$$ p_{s,r} = \PP ( \min \{ X_i \, ; \, i \leq \tau \} = r \, \, | \, \, X_1 = s ). $$
That is, the number $p_{s,r}$ is the probability that $r$ is the minimal node that the walker visits when starting from node $s$, before reaching the terminal node $N$. Clearly $\sum_{r=0}^s p_{s,r} = 1$.

\begin{lemma} For  $0 \leq r \leq s \leq N-1$,
\begin{equation}  p_{s,r} = q_r \cdot \prod_{k=r+1}^s (1 - q_k)
\label{eq_936} \end{equation}
where an empty product equals one. Moreover,  $p_{N,r} = \delta_{rN}$, where $\delta_{r N}$ is the Kronecker delta.
\label{lem_605}
\end{lemma}

\begin{proof} The expression on the right-hand side of (\ref{eq_936}) is the probability to ever reach $s-1$ when starting from $X_1 = s$, and from $s-1$ to ever reach $s-2$, etc.
until we finally reach $r$, yet from $r$ we require to {\it never} reach $r-1$. Alternatively, when $0 \leq r \leq s, s \geq 1$ we have the recurrence relation
\begin{equation}   p_{s,r} = q_s \delta_{s,r} + (1 - q_s) \cdot p_{s-1, r}. \label{eq_708} \end{equation}
This recurrence relation leads to  another proof of (\ref{eq_936}).
\end{proof}

Suppose that our random walk  $R_1,R_2,\ldots \in V$ begins at a vertex $R_1 = x$ with $d(x) = s$.
Consider a leaf $y \in \partial V$ with $dlca(x,y) = r$. What is the probability that our random walk will reach the leaf $y$?
We claim that this probability is
\begin{equation}  b_{s,r} := \PP( R_{\tau} = y ) = \sum_{k=0}^r 2^{k-N} p_{s,k}. \label{eq_534}
\end{equation}
Indeed, conditioning on the value of $k = \min_{i \leq \tau} d(R_i) $, by symmetry we know that $R_{\tau}$
is distributed uniformly among the $2^{N-k}$ leaf-descendants of the vertex $\pi_k(x)$. When $k \leq r$, exactly one of these
leaf-descendants is the leaf $y$, since the vertex of minimal depth that $(R_i)$ visits must be the vertex $\pi_k(x)$, which
is a prefix of $y$ as $k \leq r = dlca(x,y)$. Hence the probability that $R_{\tau} = y$, conditioning
on the value of $k$, equals to $1 / 2^{N-k}$ when $k \leq r$ and it vanishes otherwise. By using the definition of $p_{s,k}$ and the complete probability formula,
we obtain (\ref{eq_534}). The {\it harmonic extension operator} associated with our invariant random walk
is given by
\begin{equation}  \tilde{T} g(x) = \sum_{\omega \in \{0, 1 \}^N} b_{d(x), dlca(x,\omega)} \cdot g(\omega) \qquad \qquad \qquad (x \in V). \label{eq_1842} \end{equation}
The operator $T$ is induced from $\tilde{T}$ via formula (\ref{eq_552}) above. Requirements 1,...,4 from Section \ref{sec1}
are clearly satisfied.

\medskip We stipulate that the collection of descendants of a vertex $x \in V$, denoted by $D(x) \subseteq V$, includes the vertex $x$ itself.
 Abbreviate $a \wedge b = \min \{a,b \}$ and $a \vee b = \max \{a, b \}$.
The operator $T$ takes the form
\begin{equation} T f(x) = \sum_{y \in E} K(x,y) f(y),
\label{eq_1354} \end{equation}
where the kernel $K$ is described next.

\begin{proposition} Let $x,y \in E$ and denote $s = d(x), t = d(y), r = dlca(x,y)$. Then the following hold:
If $x \not \in D(y)$ and $y \not \in D(x)$, then
 $r \leq s \wedge t - 1$ and
\begin{equation}  K(x,y) = -q_s  \cdot 2^{-t} \cdot \sum_{k=0}^{ r} 2^{k} p_{s-1, k} \leq 0.
\label{eq_1034}
\end{equation}
Otherwise, i.e., if $y \in D(x)$ or if $x \in D(y)$ then   $r = s \wedge t$ and
\begin{equation}
K(x,y) = q_s \cdot 2^{-t} \cdot  \sum_{k=0}^{r-1} (2^r - 2^{k}) p_{s-1, k} \geq 0.
\label{eq_1037} \end{equation}
 \label{prop_1038}
\end{proposition}

\begin{proof} By (\ref{eq_552}) and (\ref{eq_1842}) we have, for any $x \in E$,
\begin{align}  \nonumber T f(x) & = \tilde{T} \tilde{f}(x) - \tilde{T} \tilde{f}(\pi_{d(x)-1} x)
\\ & =  \sum_{\omega \in \{0, 1 \}^N} b_{d(x), dlca(x,\omega)} \tilde{f}(\omega) -
\sum_{\omega \in \{0, 1 \}^N} b_{d(x)-1, dlca(\pi_{d(x)-1} x,\omega)} \tilde{f}(\omega) \nonumber
\\ & = \sum_{\omega \in \{0,1\}^N} a_{d(x), dlca(x,w)} \tilde{f}(\omega) 
\label{eq_1739}
\end{align}
where for any $0 \leq r \leq s$, by (\ref{eq_1719}) and (\ref{eq_534}),
  \begin{equation} \label{eq_1740_} a_{s,r} = b_{s,r} - b_{s-1, \min \{s-1, r \}}
= \sum_{k=0}^r 2^{k-N} p_{s,k} - \sum_{k=0}^{\min \{s-1, r \}} 2^{k-N} p_{s-1,k}. \end{equation}
Hence, by (\ref{eq_1756_}) and (\ref{eq_1739}),
\begin{align*}  T f(x) &  = \sum_{\omega \in \{0,1\}^N} a_{d(x), dlca(x,\omega)} \sum_{i=1}^{N} f(\pi_i \omega) \nonumber
\\& = \sum_{y \in E} \left[ \sum_{\omega \in \{0,1\}^N ; \pi_{d(y)} \omega = y} a_{d(x), dlca(x,\omega)} \right] f(y) = \sum_{y \in E} K(x,y) f(y),
\end{align*}
where the kernel $K$ of the operator $T$ satisfies, for any $x, y \in E$,
\begin{equation}
 K(x,y)  = \sum_{\omega \in \{0,1\}^N ; \pi_{d(y)} \omega = y} a_{d(x), dlca(x,\omega)}.
 \label{eq_1756}
\end{equation}
Fix $x, y \in E$ with $s = d(x), t = d(y)$ and $r = dlca(x,y)$.
Let us consider first the case where $x$ is {\it not} a descendant of $y$. This means
that the prefix of $y$ that is shared by $x$, is not the entire string $y$.
Hence for any $\omega \in \{0,1\}^N$ with $\pi_{d(y)} \omega = y$ we have $dlca(x,\omega) =
dlca(x,y) \leq d(y) - 1$. Therefore, from (\ref{eq_1756}) and (\ref{eq_1740_}),
\begin{align*}
 K(x,y) & = 2^{N-d(y)} a_{d(x), dlca(x,y)} &  = 2^{N- t} \left[ b_{s,r} - b_{s-1, (s-1) \wedge r} \right]
= \sum_{k=0}^r 2^{k-t} p_{s,k} -
\sum_{k=0}^{(s-1) \wedge r} 2^{k-t} p_{s-1,k},
\end{align*}
as $b_{s,r} = \sum_{k=0}^r 2^{k - N} p_{s,k}$. Since $p_{s,s} = q_s$, we have
\begin{equation}  K(x,y) =  \delta_{rs} 2^{s-t} q_s + \sum_{k=0}^{(s-1) \wedge r} 2^{k-t} \left[ p_{s,k} - p_{s-1, k} \right]
= q_s \cdot \left[ \delta_{rs} 2^{s-t}  - \sum_{k=0}^{(s-1) \wedge r} 2^{k-t} p_{s-1, k} \right],
\label{eq_1035}
\end{equation}
where we used  the relation (\ref{eq_708}), which implies that when $k \leq s-1$,
\begin{equation}  p_{s,k} - p_{s-1,k} = - q_s \cdot p_{s-1, k}. \label{eq_1758} \end{equation}
We may now prove the conclusion of the proposition in the case where $x \not \in D(y)$ and $y \not \in D(x)$.
Indeed, in this case $r \leq s \wedge t - 1$ and formula (\ref{eq_1035}) applies. Since $\delta_{rs} =0$
in this case, we deduce formula (\ref{eq_1034}) from (\ref{eq_1035}).

\medskip The next case we consider is the case where $r = s \leq t-1$, or equivalently, where $y \in D(x) \setminus \{x \}$.
Thus $x \not \in D(y)$ and formula (\ref{eq_1035}) applies. Recalling that $\sum_{k=0}^{s-1} p_{s-1,k} = 1$  we obtain from (\ref{eq_1035}) that
\begin{align*} \nonumber K(x,y) = q_s \cdot \sum_{k=0}^{s-1} (2^{s-t} - 2^{k-t}) p_{s-1, k}
= q_s \cdot 2^{-t} \cdot \sum_{k=0}^{s-1} (2^{r} - 2^{k}) p_{s-1, k}
,
\end{align*}
proving formula (\ref{eq_1037}) in the case  $y \in D(x) \setminus \{x \}$.

\medskip We move on to the case where $x \in D(y) \setminus \{ y \}$, thus $dlca(x,y) = t \leq s-1$. In this case,
by applying (\ref{eq_1756}), (\ref{eq_1740_}), (\ref{eq_534}) and then (\ref{eq_1758}),
\begin{align*}
 K(x,y)  & = \sum_{\omega \in \{0,1\}^N ; \pi_{d(y)} \omega = y} a_{d(x), dlca(x,\omega)}
 = 2^{N-d(x)} a_{d(x), d(x)} + \sum_{k=d(y)}^{d(x)-1} 2^{N-k-1} a_{d(x), k}
 \\ & = 2^{N-s} a_{s,s} + \sum_{k=t}^{s-1} 2^{N-k-1} a_{s, k}
 = 2^{N-s} (b_{s,s} - b_{s-1, s-1}) + \sum_{k=t}^{s-1} 2^{N-k-1} (b_{s, k} - b_{s-1, k})
 \\ & = 2^{N-s} \left[ \sum_{k=0}^s 2^{k-N} p_{s,k} - \sum_{k=0}^{s-1} 2^{k-N} p_{s-1,k}
 \right] + \sum_{k=t}^{s-1} 2^{N-k-1} \sum_{\ell=0}^k 2^{\ell-N} (p_{s, \ell} - p_{s-1, \ell})
 \\ & = q_{s} - q_s \sum_{k=0}^{s-1} 2^{k-s} p_{s-1, k}  - q_s \sum_{\ell=0}^{s-1} \sum_{k=\ell \vee t}^{s-1} 2^{\ell-k-1} p_{s-1, \ell}
 \\ & = q_s \left( 1 - \sum_{k=0}^{s-1} 2^{k-s} p_{s-1, k} - \sum_{\ell=0}^{s-1} [2^{\ell-\ell \vee t} - 2^{\ell-s}] p_{s-1, \ell} \right) \\ & = q_s \left( 1 - \sum_{k=0}^{s-1} 2^{k-k \vee t} p_{s-1,k} \right) = q_s \cdot \left( 1 - \sum_{k=0}^{t-1} 2^{k-t} p_{s-1,k} - \sum_{k=t}^{s-1}  p_{s-1,k} \right)
 \\ & = q_s \cdot  \sum_{k=0}^{t-1} (1 - 2^{k-t}) p_{s-1,k}=
 q_s \cdot  2^{-t} \cdot \sum_{k=0}^{t-1} (2^t - 2^{k}) p_{s-1,k}.
 \end{align*}
 Since $r = t$ in this case, we have proved formula (\ref{eq_1037}) in the case
 where $y \in D(x) \setminus \{ x \}$. Finally, the last case that remains is when $x = y$. In this case $r=s=t$ and
\begin{align*} K(x,y)  & = \sum_{\omega \in \{0,1\}^N ; \pi_{d(y)} \omega = y} a_{d(x), dlca(x,\omega)} = 2^{N-d(x)} a_{d(x), d(x)} = 2^{N-s} a_{s,s} = 2^{N-s} (b_{s,s} - b_{s-1, s-1})
\\ & = 2^{N-s} \left[ \sum_{k=0}^s 2^{k-N} p_{s,k} - \sum_{k=0}^{s-1} 2^{k-N} p_{s-1,k}
\right] = q_{s} - q_s \sum_{k=0}^{s-1} 2^{k-s} p_{s-1, k}
\\ & = q_s \cdot \sum_{k=0}^{s-1} (1 - 2^{k-s}) p_{s-1, k},
\end{align*}
 completing the proof of formula (\ref{eq_1037}).
\end{proof}

\medskip
 {\it Some examples.}
\nopagebreak
\begin{enumerate}
\item The simplest example is when $q_s = 1$ for all $s \geq 1$. In this case the operator $\tilde{T}$ is the familiar averaging operator. That is, the extension operator $\tilde{T}$ is the operator that assigns to each vertex the average of the values at the leaves of its subtree.	
In this case $$ p_{s,r} = \delta_{s,r}. $$

\item Consider the
case where the invariant random walk
is such that $X_i := d(R_i)$ is a symmetric random walk on $\{0,\ldots,N \}$, i.e.,
the probability to jump from $i$ to $i+1$ is exactly $1/2$ for $i=1,\ldots,N-1$.
Recall that $q_s$ is the probability to never leave the subtree when starting at a vertex of depth $s$.
We claim that in this example, for $s=1,\ldots,N$,
\begin{equation}  q_s = \frac{1}{N-s+1}.
\label{eq_1000} \end{equation}
Indeed, the function $f(i) = i$ is harmonic on $\{0,1,\ldots,N\}$ and hence $f(X_i)$ is a martingale.
Thus for any stopping time $\tilde{\tau}$ we have $f(X_1) = \EE f(X_{\tilde{\tau}})$.
We pick the stopping time
$$ \tilde{\tau} = \min \{ \, i ; X_i \in \{ s-1, N \} \, \} $$ and obtain (\ref{eq_1000}) since
$$ N \cdot q_s + (s-1) \cdot \left(1 - q_s \right) = s. $$
Next we use Lemma \ref{lem_605} and find a formula for $p_{s,r}$.
Since  formula (\ref{eq_1000}) is valid for any $s \geq 1$, we conclude that for any $r \geq 0$ and $s \geq r+1$,
\begin{equation}  \prod_{k=r+1}^s (1 - q_k) = \prod_{k=r+1}^s \frac{N-k}{N-k+1} = \frac{N-s}{N-r}. \label{eq_1057} \end{equation}
Formula (\ref{eq_1057}) is actually valid for any $0 \leq r \leq s$, since an empty product equals one.
Recall that $q_0 = 1$. We thus conclude from Lemma \ref{lem_605} that
for $s =0,\ldots,N-1$,
$$ p_{s,r} = \frac{N-s}{N-r} \cdot \left \{ \begin{array}{cc} 1 & r = 0 \\
\frac{1}{N-r+1} & 1 \leq r \leq s \end{array} \right. $$
while $p_{N,r} = \delta_{N,r}$.

\item Let $0 < \delta < 1$, and consider the case where  $(X_i)$ is a random
walk on $\{0, \ldots, N \}$ such that the probability to jump from $k$ to $k+1$ equals $1/2$ if $k < N-1$, and it equals $\delta$ if $k = N-1$. A harmonic function here is
$$ f(k) = \left \{ \begin{array}{cc} k & k \leq N-1 \\  N-2 + 1/\delta & k = N \end{array} \right. $$
Therefore, for $s=1,\ldots,N-1$,
$$ (N-2 + 1/\delta) \cdot q_s + (s-1) (1 - q_s) = s. $$
Thus $q_0 = q_N = 1$ while for $s=1,\ldots,N-1$,
$$ q_s = \frac{1}{N-s+1/\delta - 1}. $$
Hence for any $s \leq N-1$ and $r \leq s-1$,
$$ \prod_{k=r+1}^s (1 - q_k) = \prod_{k=r+1}^s \frac{N-k+\delta^{-1} - 2}{N-k+\delta^{-1} - 1}
= \frac{N-s+\delta^{-1} - 2}{N-r+\delta^{-1} - 2}. $$
We  conclude from Lemma \ref{lem_605} that
for $s =0,\ldots,N-1$ and $0 \leq r \leq s$,
$$ p_{s,r} = q_r \cdot \prod_{k=r+1}^s (1 - q_k)
= \frac{N-s + \delta^{-1} -2}{N-r+ \delta^{-1} -2} \cdot \left \{ \begin{array}{cc} 1 & r = 0 \\
\frac{1}{N-r+1/\delta - 1} & 1 \leq r \leq s \end{array} \right. $$
while $p_{N,r} = \delta_{N,r}$.
\end{enumerate}

We conclude this section with the following:

\begin{lemma} For any numbers $q_1,\ldots,q_{N-1} \in (0,1)$ there exists a random walk  $$ X_1,X_2,\ldots \in  \{0,\ldots, N \} $$
satisfying (\ref{eq_X}) with $ \tau = \min \{i \geq 1 \, ; \, X_i = N \}$ for $s=1,\ldots,N-1$.
\label{lem_1424}
\end{lemma}

\begin{proof} Write $x_k$ for the probability that the random walk jumps from  $k$ to $k+1$. Then for $s=0,\ldots,N-1$,
\begin{equation}  q_{s} = x_s (q_{s+1} + (1-q_{s+1}) q_s), \label{eq_944} \end{equation}
where we set $q_0 = q_N = 1$. The number $x_s \in (0,1)$ is determined by equation (\ref{eq_944})
because  $q_{s+1} + (1-q_{s+1}) q_s$ is larger than  $q_{s+1} q_s + (1-q_{s+1}) q_s = q_s > 0$.
\end{proof}

\section{The ancestral and non-ancestral parts of the kernel}
\label{sec3}

We need to bound the operator norm in $L^p(E)$ of the operator $T$ whose kernel is described in Proposition \ref{prop_1038}.
Let us consider first the {\it non-ancestral} part of the operator, given by
the kernel
\begin{equation} K_0(x,y) = K(x,y) \cdot 1_{\{ x \not \in D(y), y \not \in D(x) \}} = -1_{ \{ r \leq s \wedge t - 1 \}} \cdot q_s \cdot 2^{-t} \cdot \sum_{k=0}^r 2^k p_{s-1,k}. \label{eq_1357} \end{equation}
Here as usual
$s = d(x), t = d(y)$ and $r = dlca(x,y)$.
Write $T_0: L^p(E) \rightarrow L^p(E)$ for the operator whose kernel is $K_0$.
A function $f: E \rightarrow \RR$ is  invariant under the symmetries of the tree, or {\it invariant} in short,
if it takes the form $$ f(x)= F(d(x)) $$
for some function $F: \{1,\ldots,N \} \rightarrow \RR$.  The operator $T_0$ is equivariant under the symmetries of the tree.
Therefore, if $f(x) = F(d(x))$ is an invariant function, then so is $T_0 f$.
In fact, in the case where $f(x) = F(d(x))$ we can write
\begin{equation}  T_0 f(x) = \sum_{t=1}^N L_0(d(x),t) F(t)
\label{eq_1353} \end{equation}
for a certain kernel $L_0(s,t)$ defined for $s,t=1,\ldots,N$.

\begin{lemma} For $s,t=1,\ldots,N$,
	$$ L_0(s,t) = -q_s \cdot \sum_{k=0}^{m -1 }   (1 - 2^{k-m}) p_{s-1,k} \leq 0. $$
\label{lem_1144}
\end{lemma}

	\begin{proof} Let $r \leq \min \{t, s \}-1$.  A moment of reflection reveals that for $x \in E$ with $d(x) = s$,
		$$ n(t ;s, r) := \# \{ y \in E \,; \, d(y) = t, dlca(x, y) = r \} =
		 2^{t-r-1}.
		$$
		 By (\ref{eq_1357}), (\ref{eq_1353}) and the definition of $T_0$,
		for any $x \in E$ with $d(x) = s$,
		\begin{align*} L_0(s,t) & = \sum_{y \in E ; d(y) = t} K_0(x, y)
		 = -\sum_{r=0}^{t \wedge s -1 } n(t ; s,r) \cdot q_s \cdot 2^{-t} \cdot \sum_{k=0}^r 2^{k} p_{s-1,k}. \end{align*}
		 Denote $m = s \wedge t$. Then for  $s,t=1,\ldots,N$,
	\begin{align*}
	 L_0(s,t) &= -q_s \cdot \sum_{r=0}^{t \wedge s -1 }    \sum_{k=0}^r 2^{k-r-1} p_{s-1,k}
	 = -q_s \cdot \sum_{k=0}^{m -1 }    \sum_{r=k}^{m-1} 2^{k-r-1} p_{s-1,k}
\\&	 = -q_s \cdot \sum_{k=0}^{m -1 }   (1 - 2^{k-m}) p_{s-1,k}.
		\end{align*}
\end{proof}

Write $\Omega_N = \{1,\ldots,N \}$ and for $F: \Omega_N \rightarrow \RR$ define
\begin{equation}  \| F \|_p = \| F \|_{L^p(\Omega_N)} = \left( \sum_{k=1}^N 2^k \cdot W_k \cdot |F(k)|^p \right)^{1/p}.
\label{eq_334} \end{equation}
Observe that $\| f \|_{L^p(E)} = \| F \|_p$ if $f(x) = F(d(x))$.
Let $$ S_0 F(s) = \sum_{t=1}^N L_0(s,t) F(t) $$ so that by (\ref{eq_1353}),
\begin{equation}  T_0 (F \circ d) = (S_0 F) \circ d. \label{eq_1040} \end{equation}
For $f,g: E \rightarrow \RR$ we consider the scalar product
$$ \langle f, g \rangle = \sum_{x \in E} W_{d(x)} f(x) g(x) $$
while for $F, G: \Omega_N \rightarrow \RR$ we set
$$ \langle F, G \rangle := \langle F \circ d, G \circ d \rangle = \sum_{k=1}^N 2^k  W_{k} F(k) G(k). $$
The adjoint operators $T_0^*$ and $S_0^*$ are defined with respect to these scalar products.
The following lemma is probably well-known to experts (see, e.g.,
Howard and Schep \cite{HS} for a related argument), and its proof is provided for completeness.

\begin{lemma} Let $1 < p <\infty$. Then the norm of the operator $T_0: L^p(E) \rightarrow L^p(E)$ is attained at an invariant, non-negative function $f$, and it equals to the norm of the operator $S_0: L^p(\Omega_n) \rightarrow L^p(\Omega_n)$. \label{lem_1744}
\end{lemma}

\begin{proof} Denote momentarily $\cT = -T_0$ and $\cS = -S_0$.	
By Lemma \ref{lem_1144} the kernel $-L_0$ of the operator $\cS$ is non-negative,
and by (\ref{eq_1357}) the kernel of the operator $\cT$ is non-negative as well.
By approximation, we may assume that these two kernels are strictly positive, while
keeping condition (\ref{eq_1040}), thus
\begin{equation}  \cT (F \circ d) = (\cS F) \circ d. \label{eq_1040_} \end{equation}
We deduce that $\cT^* (F \circ d) = (\cS^* F) \circ d$.
 By compactness,
$$ \sup_{0 \not \equiv F \in L^p(\Omega_N)} \frac{\| \cS F \|_p}{\| F \|_p} $$
is attained at some function $F$. Since the kernel of $\cS$ is non-negative,
we may assume that the extremal function $F$ is non-negative.
By the Lagrange multipliers theorem, the function $F$ satisfies a certain eigenvalue equation, and in fact there exists $\lambda \in \RR$ such that
\begin{equation}  \cS^*(\cS F)^{p-1} = \lambda F^{p-1}. \label{eq_1027} \end{equation}
Since the kernel of $\cS$ is positive and $F$ is non-negative and not identically zero,
it follows from (\ref{eq_1027}) that $F$ is actually positive.
The norm of the operator $\cS: L^p(\Omega_n) \rightarrow L^p(\Omega_n)$ equals $\lambda^{1/p} > 0$, since
$$ \| \cS F \|_p^p = \langle (\cS F)^{p-1}, \cS F \rangle = \langle \cS^* (\cS F)^{p-1}, F \rangle = \lambda \langle F^{p-1}, F \rangle = \lambda \| F \|_p^p. $$
Denoting $f = F \circ d$, we find from (\ref{eq_1040_}) that $f$ is a positive invariant function satisfying
 $$ \cT^*(\cT f)^{p-1} = \lambda  f^{p-1}. $$
Since the kernel of $\cT$ is non-negative, we have the pointwise H\"older inequality
$$ \cT(uv) \leq \cT(u^p)^{1/p} \cdot  \cT(v^q)^{1/q},
$$
valid for any non-negative functions $u,v \in  L^p(E)$, where $q = p / (p-1)$.
The operator $\cT$ has a non-negative kernel, and hence its norm is attained at a
non-negative function $u \in L^p(E)$. By the pointwise H\"older inequality,
$$ (\cT u)^p \leq \cT( u^p f^{-p/q}) \cdot (\cT f)^{p/q} = \cT( u^p f^{1-p}) \cdot (\cT f)^{p-1}. $$
Therefore,
\begin{align*}  \| \cT u \|_{L^p(E)}^p  \leq \langle \cT( u^p f^{1-p}),  (\cT f)^{p-1} \rangle
= \langle  u^p f^{1-p},  \cT^* (\cT f)^{p-1} \rangle
 = \lambda  \langle  u^p f^{1-p},  f^{p-1} \rangle
= \lambda  \| u \|_{L^p(E)}^p.
\end{align*}
Thus the norm of $\cT: L^p(E) \rightarrow L^p(E)$
is at most $\lambda^{1/p}$, which is the norm of $\cS: L^p(\Omega_N) \rightarrow L^p(\Omega_N)$. The two norms must therefore
be equal, since the operator $\cS$ is equivalent to the restriction of $\cT$ to the space of invariant functions.
\end{proof}

We move on to the ancestral part of the operator, which according to Proposition \ref{prop_1038} is given by
\begin{equation} K_1(x,y) = K(x,y) - K_0(x,y) =  1_{ \{ r  = s \wedge t  \}} \cdot q_s \cdot 2^{-t} \cdot \sum_{k=0}^{r-1} (2^r - 2^k) p_{s-1,k} \geq 0. \label{eq_1740} \end{equation}
Write $T_1$ for the operator whose kernel is $K_1$. As before, for an invariant function $f(x) = F(d(x))$ we may write
\begin{equation}  T_1 f(x) = \sum_{t=1}^N L_1(d(x),t) F(t)
\label{eq_1741} \end{equation}
for a certain kernel $L_1(s,t)$ defined for $s,t=1,\ldots,N$. We also write
 $$ S_1 F(s) = \sum_{t=1}^N L_1(s,t) F(t). $$
It is possible to use formula (\ref{eq_552}) for the operator
$T$ and the definition (\ref{eq_1842}) of the harmonic extension operator
$\tilde{T}$ and deduce that
\begin{equation} S_0 + S_1 \equiv 0, \label{eq_114} \end{equation}
essentially because the only invariant, harmonic function on the vertices of the tree is the constant function.
An alternative, more direct proof of (\ref{eq_114}) is provided
in the following:

\begin{lemma} Let $1 < p < \infty$. Then the norm of
the operator $T_1: L^p(E) \rightarrow L^p(E)$ is equal to the norm of $S_1: L^p(\Omega_n) \rightarrow L^p(\Omega_n)$.
Additionally, for $s,t=1,\ldots,N$ with $m = s \wedge t$ we have
	$$ L_1(s,t) =
	 q_s \cdot \sum_{k=0}^{m-1} (1 - 2^{k-m}) p_{s-1,k} = -L_0(s,t). $$
	\end{lemma}

\begin{proof} The first assertion of the lemma follows from the fact that the kernel of $T_1$ is non-negative and invariant under the symmetries of the tree, as in Lemma \ref{lem_1744}. For the second part, let $s,t=1,\ldots,N$ and denote $m = s \wedge t$.
	We claim that for $x \in E$ with $d(x) = s$,
	\begin{equation} n(t ;s, m) := \# \{ y \in E \, ; \, d(y) = t, dlca(x, y) = m \} = \max \{1, 2^{t-s} \}.
	\label{eq_1748}
	\end{equation}
	Indeed, assume first that $t \geq s$. How many $y$'s are there with $d(y) = t$ and $dlca(x, y)= m$? Since $m = s$, the answer is $2^{t-s}$. Next, if $s \geq t$, then the number of such $y$'s is one. This proves (\ref{eq_1748}). Therefore,
	\begin{align*} L_1(s,t) & = \sum_{y \in E ; d(y) = t} K_1(x, y)
		= n(t ; s,m) \cdot  q_s \cdot 2^{-t} \cdot \sum_{k=0}^{m-1} (2^m - 2^k) p_{s-1,k}
		\\ & = \max \{2^{-t}, 2^{-s} \} \cdot q_s \cdot \sum_{k=0}^{m-1} (2^m - 2^k) p_{s-1,k}.
		\\ & = 2^{-m} \cdot q_s \cdot \sum_{k=0}^{m-1} (2^m - 2^k) p_{s-1,k}.
		\end{align*}	
\end{proof}

\begin{corollary} We have
$$ \| T \|_{L^p(E) \rightarrow L^p(E)} \leq 2 \| S_0 \|_{L^p(\Omega_N) \rightarrow L^p(\Omega_N)}. $$
\label{cor_1057}
\end{corollary}

\begin{proof} This follows from the fact that $T = T_0 + T_1$ together with the
facts that $\| T_0 \| = \| S_0 \|$ and $\| T_1 \| = \| S_1 \|$ while $S_1 = - S_0$.
\end{proof}

In view of Corollary \ref{cor_1057}, we are interested in bounds for the norm of the operator $S = -S_0 = S_1: L^p(\Omega_N) \rightarrow L^p(\Omega_N)$ whose non-negative kernel is
\begin{equation}  L(s,t) = q_s \cdot \sum_{k=0}^{s \wedge t-1} (1 - 2^{k-s \wedge t}) p_{s-1,k} \leq q_s \cdot \sum_{k=0}^{s \wedge t-1} p_{s-1,k}.
\label{eq_1126}
\end{equation}
From Lemma \ref{lem_605} we know that $p_{s,r} = q_r \cdot \prod_{k=r+1}^s (1 - q_k)$ for $s \leq N-1$. A little
exercise in probability shows that for any $1 \leq m \leq s \leq N$,
\begin{equation}  \sum_{k=0}^{m-1} p_{s-1,k} = \prod_{k=m}^{s-1} (1 - q_k).
\label{eq_1129} \end{equation}
Alternatively, (\ref{eq_1129}) holds true for $m = 0$ as $q_0 = 1$, and it may be proven by induction on $m$ since
$$  p_{s-1,m} + \prod_{k=m}^{s-1} (1 - q_k)  = q_m \cdot \prod_{k=m+1}^{s-1} (1 - q_k) + \prod_{k=m}^{s-1} (1 - q_k)  = \prod_{k=m+1}^s (1 - q_k). $$
From (\ref{eq_1126}) and (\ref{eq_1129}) we thus obtain

\begin{corollary} For $s,t=1,\ldots,N$, with $m = \min \{s, t \}$,
$$ 0 \leq L(s,t) \leq q_s \prod_{k=m}^{s-1} (1 - q_k), $$
where an empty product equals one.
\label{cor_209}
\end{corollary}

\bigskip
 {\it Some examples (parallel to the ones discussed in Section \ref{sec2}).}
\begin{enumerate}
	\item For the averaging operator, where $q_s = 1$ and $ p_{s,r} = \delta_{s,r}$, we have
	$$ L(s,t) = -\frac{1}{2} \cdot 1_{\{ s \leq t \}}, $$
i.e., this is the matrix whose entries equal $0$ below the diagonal and $-1/2$ on and above the diagonal.
	This is a rather simple matrix, and it is bounded with respect to the weighted $L_p$-norm for quite a few sequences of weights.
\item For the symmetric random walk matrix, we have $q_0 = 1$ while
for $1 \leq s \leq N$,
$$ q_s = \frac{1}{N-s+1}.
$$
Hence in view of Corollary \ref{cor_209}, with $m = \min \{s, t \}$,
$$ 0 \leq L(s,t) \leq \frac{1}{N-s+1} \prod_{k=m}^{s-1} \frac{N-k}{N-k+1} = \frac{1}{N-m+1}. $$
\item In the case where
$$ q_s = \frac{1}{N-s+1/\delta - 1} $$
for some $0 < \delta < 1$, we have
$$ 0 \leq L(s,t) \leq \frac{1}{N-s+1/\delta-1} \prod_{k=m}^{s-1} \frac{N-k+1/\delta-2}{N-k+1/\delta-1} = \frac{1}{N-m+1/\delta-1}. $$
\end{enumerate}

All that remains is to bound the $L_p(\Omega_N)$-norm of the operator whose kernel is discussed in Corollary \ref{cor_209}.
Recall from Lemma \ref{lem_1424} that we have the freedom to choose the parameters $q_1,\ldots,q_{N-1} \in (0,1)$ as we please.

\medskip How should we choose these parameters? Since $L(N, t) \leq \prod_{k=t}^{N-1} (1 - q_k)$ and we are looking for upper bounds for the norm,
the $q_s$ should not be too tiny. On the other hand, $L(s,t) \leq q_s$ for $s \leq N-1$
and hence it is beneficial to choose $q_s$ rather small. We would therefore need some balance for the $q_s$,
which is the subject of the next section.

\section{One-dimensional analysis}
\label{sec4}

Let $1 < p < \infty$.
It will be slightly more convenient to denote
$$ \cK(s,t) = L(N+1-s, N+1-t) \qquad \text{and} \qquad Q_s = q_{N+1-s}. $$
Recalling from (\ref{eq_322}) that $q_N = 1$, we see that
$$ Q_1 = 1. $$
From Corollary \ref{cor_209}
we know that for $s,t=1,\ldots,N$,
\begin{align*}
\cK(s,t)  = L(N+1-s, N+1-t) \leq q_{N+1-s} \prod_{k=N+1-\max \{s,t \}}^{N-s} (1 - q_k)
 = Q_s \prod_{k=s+1}^{\max \{s, t \}} (1 - Q_k). \end{align*}
 From Corollary \ref{cor_209} we know that $L \geq 0$. Consequently, for $s,t=1,\ldots,N$,
\begin{equation}  0 \leq \cK(s,t) \leq \left \{ \begin{array}{cc} Q_s & t \leq s \\ Q_s \cdot \prod_{k=s+1}^t (1-Q_k)  & t \geq s+1 \end{array} \right. \label{eq_1209}
\end{equation}
Recall that we are given edge weights $W_1,\ldots,W_N > 0$, and that the associated $L_p(E)$-norm
is given by (\ref{eq_346}). Denote  $$ w_s := 2^{N+1-s} \cdot W_{N+1-s} > 0. $$
Consider the
weighted $\ell_p$-norm
\begin{equation}  \| f \|_{p, w} = \left( \sum_{k=1}^N w_k |f(k)|^p \right)^{1/p} \label{eq_1307}
\end{equation}
and the operator $\cT$ whose kernel is $\cK(s,t)$.
We are allowed to choose the weights $q_1,\ldots,q_{N-1} \in (0,1)$ as we please, or equivalently,
we have the freedom to determine $Q_2,\ldots,Q_N \in (0,1)$. We must keep $Q_1 = 1$.
Based on considerations related to the {\it Muckenhoupt criterion} discussed below, we set
\begin{equation}  Q_s = \frac{w_s^{-1/(p-1)}}{\sum_{k=1}^s w_k^{-1/(p-1)}}. \label{eq_1306}
\end{equation}
It is clear that $Q_1 = 1$ and that $Q_s \in (0,1)$ for all $s \geq 2$. Recall that $q = p / (p-1)$.

\begin{lemma} In order to prove Theorem \ref{thm_1226}, it suffices
to show that the operator norm of $\cT$ with respect
to the $\| \cdot \|_{p, w}$-norm is bounded by a constant $\hat{C}_p > 0$ depending only on $p$, where in fact
\begin{equation}  \hat{C}_p \leq
2 p^{1/p} q^{1/q} \cdot \left( 1 + \max \{ (p-1)^{-1/p}, (q-1)^{-1/q}  \} \right). \label{eq_355} \end{equation}
\label{lem_350}
\end{lemma}

\begin{proof} In view of Corollary \ref{cor_1057}, it suffices to
bound the operator norm of $-S_0$, whose kernel is $L$, with respect to  the $L^p(\Omega_N)$-norm defined  in (\ref{eq_334}).
Under the transformation $$ s \mapsto N+1-s $$ the operator $-S_0$ whose kernel is $L$ transforms to the operator
$\cT$ whose kernel is $\cK$. The $L^p(\Omega_N)$-norm from (\ref{eq_334}) transforms to the
$\| \cdot \|_{p,w}$-norm defined in (\ref{eq_1307}). Hence Theorem \ref{thm_1226} would follow
once we obtain the bound (\ref{eq_355}), where $\bar{C}_p \leq 2 \hat{C}_p$ by Corollary \ref{cor_1057}.
\end{proof}

The remainder of this section is devoted to the proof of the following:

\begin{proposition} The operator norm of $\cT$ with respect to the norm (\ref{eq_1307}) is bounded
by a number $\hat{C}_p$ depending only on $p \in (1,\infty)$. In fact, we have the bound (\ref{eq_355})
for the constant $\hat{C}_p$.
\label{prop_1152}	
\end{proposition}

Our main tool in the proof of Proposition \ref{prop_1152} is the Muckenhoupt criterion \cite{M}, which is an indispensable tool
for proving one-dimensional inequalities of Poincar\'e-Sobolev type.
For the reader's convenience, we include here a statement and a proof of a straightforward modification of the  Muckenhoupt criterion, with sums in place of integrals:

\begin{theorem}{(Muckenhoupt)} Let $1 < p < \infty$ and write $\Omega_N = \{1,\ldots, N\}$. Let $U, V: \Omega_N \rightarrow (0, \infty)$
and let $A > 0$ be such that for all $r =1,\ldots,N$,
 \begin{equation}  \left( \sum_{k=r}^N |U(k)|^p \right)^{1/p} \leq A \left( \sum_{k=1}^{r} |V(k)|^{-q} \right)^{-1/q}. \label{eq_1325} \end{equation}
 Then for any function $f: \Omega_N \rightarrow \RR$,
 \begin{equation}  \left( \sum_{k=1}^{N} \left| U(k) \sum_{\ell=1}^k f(\ell) \right|^p  \right)^{1/p} \leq C_p A \left( \sum_{k=1}^{N} |V(k) f(k)|^p \right)^{1/p}, \label{eq_1326} \end{equation}
 with $C_p = p^{1/p} q^{1/q}$.
 \label{thm_1459}
\end{theorem}

By continuity, the analog of Theorem \ref{thm_1459} for $p = 1,\infty$ holds true with $C_1 = C_{\infty} = 1$.
We remark that as in \cite{M}, this  criterion  is tight, in the sense that the infimum over all $A > 0$ satisfying (\ref{eq_1325}) is equivalent to the best constant
in inequality (\ref{eq_1326}). For the proof of Theorem \ref{thm_1459} we require the following:

\begin{lemma} For  $1 < p < \infty, q = p / (p-1), \alpha_1,\ldots,\alpha_N > 0$ and $r=1,\ldots,N$,
	\begin{equation}  \sum_{k=1}^r \alpha_k \left( \sum_{\ell=1}^k
	\alpha_{\ell} \right)^{-1/p} \leq q \cdot \left( \sum_{k=1}^{r} \alpha_k \right)^{1/q}.
	\label{eq_1505} \end{equation}
	\label{lem_1539}
\end{lemma}

\begin{proof}
	We will use the simple  inequality
	\begin{equation} (a+b)^{1/q} - a^{1/q} \geq b \cdot \min_{\xi \in (a, a+b)} \frac{\xi^{1/q-1}}{q} = \frac{1}{q} (a+b)^{1/q-1} \cdot b, \label{eq_1429_} \end{equation}
	valid for any $a, b \geq 0$ with $a+b > 0$.  From (\ref{eq_1429_}), for  $k=1,\ldots,N$,
	$$ \left(\sum_{\ell=1}^{k} \alpha_{\ell} \right)^{1/q}  - \left(\sum_{\ell=1}^{k-1} \alpha_{\ell} \right)^{1/q} \geq \frac{1}{q}
	\cdot \alpha_k \cdot \left(\sum_{\ell=1}^k \alpha_{\ell} \right)^{1/q-1}, $$
	where an empty sum equals zero. By summing this  for $k=1,\ldots, r$ we obtain (\ref{eq_1505}). \end{proof}

\begin{proof}[Proof of Theorem \ref{thm_1459} (Muckenhoupt)] By the H\"older inequality, for any function $f: \Omega_N \rightarrow \RR$ and  weights $h: \Omega_N \rightarrow (0, \infty)$,
	\begin{align} \nonumber  \sum_{k=1}^{N} \left| U(k) \sum_{\ell=1}^k f(\ell) \right|^p   & \leq \sum_{k=1}^N U^p(k) \cdot \sum_{\ell=1}^k |f(\ell) V(\ell) h(\ell)|^p \cdot \left( \sum_{j=1}^k |V(j) h(j)|^{-q} \right)^{p/q} \\
	& = \sum_{\ell=1}^N |f(\ell) V(\ell) h(\ell)|^p \sum_{k=\ell}^N U^p(k) \left( \sum_{j=1}^k |V(j) h(j)|^{-q} \right)^{p/q}. \label{eq_1447}
	\end{align}	
Set $h(k) = \left( \sum_{\ell=1}^k V(\ell)^{-q} \right)^{1/(pq)}$. By applying
Lemma \ref{lem_1539} with $\alpha_k = V(k)^{-q}$ we obtain
$$ \sum_{k=1}^r |V(k) h(k)|^{-q} = \sum_{k=1}^r \alpha_k \left(\sum_{\ell=1}^k \alpha_\ell \right)^{-1/p} \leq q \cdot \left( \sum_{k=1}^{r} \alpha_k \right)^{1/q}  =
q \cdot \left( \sum_{k=1}^{r} V(k)^{-q}  \right)^{1/q}.
$$
Hence for any $f: \Omega_N \rightarrow \RR$, the expression in (\ref{eq_1447}) is at most
\begin{equation}
q^{p/q} \cdot \sum_{\ell=1}^N |f(\ell) V(\ell) h(\ell)|^p \sum_{k=\ell}^N U^p(k) \left( \sum_{j=1}^k |V(j)|^{-q} \right)^{p/q^2}.
\label{eq_513} \end{equation}
By applying (\ref{eq_1325}) and then Lemma \ref{lem_1539} with $\alpha_k = U^p(N+1-k)$ and with $p \in (1, \infty)$ playing the r\^ole of $q \in (1, \infty)$, we see that
\begin{align*}
\sum_{k=\ell}^N & |U(k)|^p  \left( \sum_{j=1}^k |V(j)|^{-q} \right)^{p/q^2}
\leq A^{p/q} \sum_{k=\ell}^N |U(k)|^p \left( \sum_{j=k}^N |U(j)|^{p} \right)^{-1/q}  \\ & = A^{p/q} \sum_{k=1}^{N+1-\ell} \alpha_k \left( \sum_{j=1}^k \alpha_j \right)^{-1/q} \leq p \cdot A^{p/q} \left( \sum_{k=1}^{N+1-\ell} \alpha_k \right)^{1/p} =
p \cdot A^{p/q} \left( \sum_{k=\ell}^{N} U^p(k) \right)^{1/p}.
\end{align*}
Hence the expression in (\ref{eq_513}) is at most
\begin{align*}
p \cdot (q A)^{p/q} \cdot \sum_{\ell=1}^N |f(\ell) V(\ell) h(\ell)|^p \cdot \left( \sum_{k=\ell}^N U^p(k) \right)^{1/p}.
\end{align*}
Applying (\ref{eq_1325}) again we bound the last expression from above by
$$ p \cdot (q A)^{p/q} \cdot A \cdot \sum_{\ell=1}^N |f(\ell) V(\ell) h(\ell)|^p \cdot \left( \sum_{k=1}^{\ell} V^{-q}(k) \right)^{-1/q}
= p q^{p/q} A^p\sum_{\ell=1}^N |f(\ell) V(\ell) |^p, $$
completing the proof.
\end{proof}

\begin{corollary} Let $1 < p < \infty$ and $\Omega_N = \{1,\ldots, N\}$. Let $U, V: \Omega_N \rightarrow (0, \infty)$ and let $A > 0$ be such that for all $r =1,\ldots,N$,
	\begin{equation}  \left( \sum_{k=1}^r |U(k)|^p \right)^{1/p} \leq A \left( \sum_{k=r}^{N} |V(k)|^{-q} \right)^{-1/q}. \label{eq_1325_} \end{equation}
	Then for any $f: \Omega_N \rightarrow \RR$,
	\begin{equation}  \left( \sum_{k=1}^{N} \left| U(k) \sum_{\ell=k}^N f(\ell) \right|^p  \right)^{1/p} \leq C_p A \left( \sum_{k=1}^{N} |V(k) f(k)|^p \right)^{1/p}, \label{eq_1326_} \end{equation}
	with $C_p = p^{1/p} q^{1/q}$.
	\label{cor_1459_}
\end{corollary}

\begin{proof} Denote $\tilde{U}(k) = U(N+1-k)$ and $\tilde{V}(k) = V(N+1-k)$. Then from (\ref{eq_1325_}),
	for all $r =1,\ldots,N$,
	$$ \left( \sum_{k=r}^N |\tilde{U}(k)|^p \right)^{1/p} \leq A \left( \sum_{k=1}^{r} |\tilde{V}(k)|^{-q} \right)^{-1/q}. $$
By Theorem \ref{thm_1459}, this implies that for any $g: \Omega_N \rightarrow \RR$, denoting
$f(r) = g(N+1-r)$,
$$ \left( \sum_{k=1}^{N} \left| \tilde{U}(k) \sum_{\ell=1}^k f(N+1-\ell) \right|^p  \right)^{1/p} \leq C_p A \left( \sum_{k=1}^{N} |\tilde{V}(k) f(N+1-k)|^p \right)^{1/p},  $$
or equivalently,
$$ \left( \sum_{k=1}^{N} \left| \tilde{U}(N+1-k) \sum_{\ell=k}^N f(\ell) \right|^p  \right)^{1/p} \leq C_p A \left( \sum_{k=1}^{N} |\tilde{V}(N+1-k) f(k)|^p \right)^{1/p}. $$
	This implies (\ref{eq_1326_}). \end{proof}

\begin{proposition} For $f: \Omega_n \rightarrow \RR$ and $s=1,\ldots,N$ denote
	$$ \cT_0 f(s) = Q_s \sum_{t=1}^s f(t), $$
where $Q_s$ is defined in (\ref{eq_1306}) above.
Then the operator norm of $\cT_0$ with respect to the norm $\| \cdot \|_{p,w}$ defined in (\ref{eq_1307}) is bounded
by a number $ \tilde{C}_p$ depending only on $p \in (1,\infty)$. In fact, $\tilde{C}_p \leq 2^{1/p} p^{1/p} q^{1/q} \cdot \max \{ 1, (p-1)^{-1/p} \}$. \label{prop_1622}	
\end{proposition}

The proof of Proposition \ref{prop_1622} requires the following:

\begin{lemma} For $1 < p < \infty, \alpha_1,\ldots,\alpha_N > 0$ and $r=1,\ldots,N$, 	
	\begin{equation}  \sum_{k=r}^N \alpha_k \left( \sum_{\ell=1}^k
	\alpha_{\ell} \right)^{-p} \leq \frac{2}{\min \{1, p-1\}} \cdot \left( \sum_{k=1}^{r} \alpha_k \right)^{-p+1}.
	\label{eq_1505_} \end{equation}
	\label{lem_1609_}
\end{lemma}

\begin{proof} We use the inequality
	$$ (p-1) \cdot b (a+b)^{-p} \leq a^{-p+1} - (a+b)^{-p+1}, $$
	 which is valid for any $a, b  > 0$.  Then for  $k=2,\ldots,N$,
	$$ (p-1) \cdot \alpha_k  \left( \sum_{\ell=1}^k
	\alpha_{\ell} \right)^{-p}  \leq \left( \sum_{\ell=1}^{k-1} \alpha_\ell \right)^{-p+1}  - \left( \sum_{\ell=1}^{k} \alpha_\ell \right)^{-p+1}. $$
 By summing this  for $k=r+1,\ldots, N$ we obtain
	$$
	(p-1) \cdot \sum_{k=r+1}^{N} \alpha_k \left( \sum_{\ell=1}^k
	\alpha_{\ell} \right)^{-p} \leq  \left( \sum_{k=1}^{r} \alpha_k
	\right)^{-p+1}  - \left( \sum_{\ell=1}^{N} \alpha_\ell \right)^{-p+1}
\leq \left( \sum_{k=1}^{r} \alpha_k
	\right)^{-p+1},
	$$
where an empty sum equals zero.
We conclude (\ref{eq_1505_}) by summing this with the trivial inequality
$$  \min \{1, p-1 \} \cdot \alpha_r 	\left( \sum_{\ell=1}^r
\alpha_{\ell} \right)^{-p} \leq \left( \sum_{k=1}^{r} \alpha_k \right)^{-p+1}. $$
\end{proof}

\begin{proof}[Proof of Proposition \ref{prop_1622}] Define
	$$ U(k) = w_k^{1/p} \cdot Q_k \qquad \text{and} \qquad V(k) = w_k^{1/p}. $$
	Let us verify the condition of the Muckenhoupt criterion. We need to find $A > 0$ such that 	
	for all $r =1,\ldots,N$ inequality (\ref{eq_1325}) holds true, that is,
\begin{equation}   \sum_{k=r}^N w_k Q_k^p \leq A^p \left( \sum_{k=1}^{r} w_k^{-q/p}  \right)^{-p/q}. \label{eq_1342} \end{equation}
Recall that $p/q = p-1$. From the definition (\ref{eq_1306}) of $Q_k$, we need
$$  \sum_{k=r}^N w_k^{-1/(p-1)} \left( \sum_{\ell=1}^k w_\ell^{-1/(p-1)} \right)^{-p} \leq A^p \left( \sum_{k=1}^{r} w_k^{-1/(p-1)}  \right)^{-(p-1)}. $$
Setting $\alpha_k = w_k^{-1/(p-1)}$ and using Lemma \ref{lem_1609_}, we see that
(\ref{eq_1342}) holds true with $$ A = 2^{1/p} \cdot \max \{ 1, (p-1)^{-1/p} \}. $$
From Theorem \ref{thm_1459} we thus conclude that for any $f: \Omega_N \rightarrow \RR$,
$$ \left( \sum_{k=1}^{N} w_k \left| Q_k \sum_{\ell=1}^k f(\ell) \right|^p  \right)^{1/p} \leq p^{1/p} q^{1/q} A \cdot \left( \sum_{k=1}^{N} w_k |f(k)|^p \right)^{1/p}. $$
This implies the required bound for the operator norm of $f$.
\end{proof}

\begin{proposition} For $f: \Omega_n \rightarrow \RR$ and $s=1,\ldots,N$ denote
	$$ \cT_1 f(s) =  \sum_{t=s+1}^N \left[ Q_s \prod_{k=s+1}^t (1-Q_k) \right] f(t), $$
where $Q_s$ is as in (\ref{eq_1306}) above.
Then the operator norm of $\cT_1$ with respect to the norm $\| \cdot \|_{p, w}$ defined in (\ref{eq_1307}) is bounded
	by $2^{1/q} p^{1/p} q^{1/q} \cdot  \max \{1, (q-1)^{-1/q} \}$. \label{prop_1746}	
\end{proposition}

\begin{proof} Denote $\alpha_k = w_k^{-1/(p-1)}$ and recall from (\ref{eq_1306}) that
$Q_k = \alpha_k / \sum_{\ell=1}^k \alpha_{\ell}$. We claim that for any $t \geq s+1$,
\begin{equation}
 Q_s \prod_{k=s+1}^t (1-Q_k) = \frac{\alpha_s}{\sum_{k=1}^t \alpha_k}.
 \label{eq_1811}
\end{equation}
 Indeed, (\ref{eq_1811}) holds true for
$t=s$, since  an empty product equals one, and for $t \geq s+1$ it is proven by an easy induction on $t$. Consequently,
$$ \cT_1 f(s) = \alpha_s \sum_{t=s+1}^N \frac{1}{\sum_{j=1}^t \alpha_j} f(t). $$
Since $p,q \in (1, \infty)$, the elementary inequality of Lemma \ref{lem_1609_}
is valid also when $p$ is replaced by $q$. It implies that for $r=1,\ldots, N$,
\begin{equation}
\left( \sum_{k=1}^r \alpha_k \right)^{-q+1} \geq \frac{\min \{1, q-1 \}}{2} \cdot \sum_{k=r}^N \alpha_k \left(  \sum_{j=1}^k \alpha_j
\right)^{-q}.
\label{eq_1156}
\end{equation}
Set $A = 2^{1/q} \min \{1, q-1 \}^{-1/q}$. Since $\alpha_k = w_k^{-1/(p-1)}$ and $q/p = q-1 = 1/(p-1)$, it follows from (\ref{eq_1156}) that
for $r=1,\ldots,N$,
$$ \left( \sum_{k=1}^r w_k \alpha_k^p \right)^{1/p} \leq A \left( \sum_{k=r}^N w_k^{-q/p} \left(  \sum_{j=1}^k \alpha_j
\right)^{-q} \right)^{-1/q}. $$
This is precisely the Muckenhoupt criterion from Corollary
\ref{cor_1459_}, with
$$ U(k) = w_k^{1/p} \alpha_k \qquad \text{and} \qquad V(k) =
w_k^{1/p} \sum_{j=1}^k \alpha_j. $$
Thus, by Corollary \ref{cor_1459_}, for any  $g: \Omega_N \rightarrow \RR$,
\begin{equation}
\sum_{k=1}^N w_k
\left|\alpha_k \sum_{\ell=k}^N g(\ell) \right|^p
\leq (C_p A)^p \sum_{k=1}^N w_k
\left| \left( \sum_{j=1}^k \alpha_j \right) g(k) \right|^p,
\label{eq_1203} \end{equation}
with $C_p = p^{1/p} q^{1/q}$. By restricting attention to non-negative functions $g$ in
(\ref{eq_1203}), we may alter (\ref{eq_1203}) and replace $\sum_{\ell=k}^N$ by the shorter sum $\sum_{\ell=k+1}^N$. Inequality (\ref{eq_1203}) remains correct, for non-negative $g$,
also after this modification.
Denoting $g(k) = f(k) / \sum_{j=1}^k \alpha_j$, we conclude that for any non-negative
function $f: \Omega_N \rightarrow \RR$,
\begin{equation}
\sum_{k=1}^N w_k
\left|\alpha_k \sum_{\ell=k+1}^N \frac{1}{\sum_{j=1}^\ell \alpha_j} f(\ell) \right|^p
\leq (C_p A)^p \sum_{k=1}^N w_k
\left| f(k) \right|^p.
\label{eq_1207} \end{equation}
Since the kernel of $\cT_1$ is non-negative, its operator norm is attained
at a non-negative function $f: \Omega_N \rightarrow \RR$.  Therefore
(\ref{eq_1207}) implies the required bound for the  operator norm of $\cT_1$.
\end{proof}

\begin{proof}[Proof of Proposition \ref{prop_1152}]
The kernel of the operator $\cT$ is given in (\ref{eq_1209}). It is a non-negative
kernel, and therefore the operator norm of $\cT$ is at most the operator norm of the operator
whose kernel is the expression on the right-hand side of (\ref{eq_1209}). The latter operator equals
$$ \cT_0 + \cT_1 $$
with $\cT_0$ from Proposition \ref{prop_1622} and $\cT_1$ from Proposition \ref{prop_1746}.
From these two propositions it follows that the operator norm of $\cT$ is at most
\begin{align*} 2^{1/p} p^{1/p} q^{1/q} \cdot & \max \{ 1, (p-1)^{-1/p} \} + 2^{1/q} p^{1/p} q^{1/q} \cdot  \max \{1, (q-1)^{-1/q} \}. \end{align*}
\end{proof}

Theorem \ref{thm_1226} follows from Lemma \ref{lem_350} and Proposition \ref{prop_1152}.

\bigskip
\newpage \noindent {\it Remarks.}\nopagebreak
\begin{enumerate}
\item In this paper we have left open several natural questions, including
the existence of linear extension operators for $\dot{W}^{1,p}(T)$ for
weighted trees $T$ in the extreme cases $p=1, p=\infty$, as well as the
analog of our result for the inhomogeneous Sobolev space $W^{1,p}(T)$ in
place of $\dot{W}^{1,p}(T)$.
\item The problem of existence of linear Sobolev extension operators for
weighted trees arose in connection with an extension problem for
$W^{2,p}(\RR^2)$. More precisely, given $E \subseteq \RR^2$, let $\dot{W}^{2,p}(E)$
denote the space of restrictions to $E$ of functions in $\dot{W}^{2,p}(\RR^2)$,
endowed with the natural seminorm. Does there exist a linear extension
operator from $\dot{W}^{2,p}(E)$  to $\dot{W}^{2,p}(\RR^2)$? The answer is
affirmative for $p > 2$; see A. Israel \cite{I}. For $1<p<2$,
the answer is unknown. For a particular class of examples $E$, the
problem reduces to the question answered by Theorem
\ref{thm_1226_}.
\end{enumerate}


\begin{thebibliography}{99}
	

\bibitem{BBGS} Bj\"orn, A., Bj\"orn, J., Gill, J. T., Shanmugalingam, N.,
{\it Geometric analysis on Cantor sets and trees.} J. Reine Angew. Math., Vol. 725, (2017), 63--114.

\bibitem{HS} Howard, R., Schep, A. R.,
{\it Norms of positive operators on $L^p$-spaces.}
Proc. Amer. Math. Soc., Vol. 109, no. 1, (1990), 135--146.

\bibitem{I} Israel, A., {\it  A bounded linear extension operator for $L^{2,p}(\RR^2)$.} Ann. of Math. (2), Vol. 178, no. 1, (2013), 183--230.


\bibitem{M} Muckenhoupt, B., {\it Hardy's inequality with weights}. Studia Math., Vol. 44, (1972), 31--38.

\bibitem{Shvartsman} Shvartsman, P., {\it Sobolev $W^1_p$-spaces on closed subsets of $\RR^n$}.
Adv. Math., Vol. 220, no. 6, (2009), 1842--1922.

\end{thebibliography}
\end{document}